\documentclass[12pt]{amsart}

\usepackage{tikz}
\usepackage{xcolor}

\usepackage{latexsym,amsmath,amssymb,epic}
\usepackage{latexsym,amssymb,epic}
\usepackage{amsmath,amsthm}
\usepackage{color}

\usepackage{soul}

\usepackage{hyperref}
\hypersetup{
	colorlinks=true, 
	linktoc=all, 
	linkcolor=blue} 

\newtheorem{theorem}{Theorem}[section]

\begin{document}

	\title[On the Parity of  $\phi_k(n)$]{On the Parity of the Generalized Frobenius Partition Functions $\phi_k(n)$}

	\author{George E. Andrews}
	\address{Department of Mathematics, The Pennsylvania State University, University Park, PA 16802, USA}
	\email{gea1@psu.edu}

	\author{James A. Sellers}
	\address{Department of Mathematics and Statistics, University of Minnesota Duluth, Duluth, MN 55812, USA}
	\email{jsellers@d.umn.edu}

    \author{Fares Soufan}
    \address{Department of Mathematics and Statistics, University of Minnesota Duluth, Duluth, MN 55812, USA}
	\email{soufa005@umn.edu}

	
	
	\maketitle
	
 \begin{abstract}
 In his 1984 Memoir of the American Mathematical Society, George Andrews defined two families of functions, $\phi_k(n)$ and $c\phi_k(n),$ which enumerate two types of combinatorial objects which Andrews called generalized Frobenius partitions.  As part of that Memoir, Andrews proved a number of Ramanujan--like congruences satisfied by specific functions within these two families.  In the years that followed, numerous other authors proved similar results for these functions, often with a view towards a specific choice of the parameter $k.$  In this brief note, our goal is to identify an {\bf infinite} family of values of $k$ such that $\phi_k(n)$ is even for all $n$ in a specific arithmetic progression; in particular, our primary goal in this work is to prove that, for all positive integers $\ell,$ all primes $p\geq 5,$ and all values $r,$ $0 < r < p,$ such that $24r+1$ is a quadratic nonresidue modulo $p,$ 
 $$
 \phi_{p\ell-1}(pn+r) \equiv 0 \pmod{2}
 $$
 for all $n\geq 0.$
 Our proof of this result is truly elementary, relying on a lemma from Andrews' Memoir, classical $q$--series results, and elementary generating function manipulations. Such a result, which holds for infinitely many values of $k,$ is rare in the study of arithmetic properties satisfied by generalized Frobenius partitions, primarily because of the unwieldy nature of the generating functions in question. \end{abstract}

\section{Introduction}
In his 1984 AMS Memoir, George Andrews 
\cite{AndMem} 
defined two families of combinatorial objects known as {\it generalized Frobenius partitions}.  These are generalizations of the two--rowed arrays, often known as Frobenius symbols, which arise from considering the rows and columns of the Ferrers graph of an ordinary partition once the Durfee square has been ``removed''.  In the process, Andrews defined two families of functions, $\phi_k(n)$ and $c\phi_k(n),$ as the number of generalized Frobenius partitions of weight $n$ in these two families of objects, respectively.  In 
\cite{AndMem}, Andrews studied these functions $\phi_k(n)$ and $c\phi_k(n)$ from several perspectives, including proving a number of Ramanujan--like congruences satisfied by these functions.  This, in turn, led a number of others to extend Andrews' congruence results.  

While there exists an extensive literature on the subject of congruences satisfied by generalized Frobenius partition functions, our focus in this note will be on parity results.  We highlight here that a number of authors have proven congruence results with even moduli for these functions; see, for example, the work of 
Andrews 
\cite[Theorem 10.2]{AndMem}, 
Baruah and Sarmah 
\cite{BarSar1, BarSar2},   
Chan, Wang, and Yang 
\cite{ChanWangYang}, 
Cui and Gu, 
\cite{CuiGu}, 
Cui, Gu, and Huang 
\cite{CuiGuHuang},  
and 
Jameson and Wieczorek 
\cite{JamWie}
where specific congruence results with even moduli are proved.  Several additional papers involving congruence results for generalized Frobenius partitions, but with odd moduli, also appear in the literature.  

What is striking about many of the works cited above is that the authors focus specifically on a particular value of the parameter $k$ in order to manipulate the generating function in question to prove their results.  One exception to this rule is Andrews' Theorem 10.2 in 
\cite[Theorem 10.2]{AndMem}:  

\begin{theorem}
Let $p$ be prime and and let $r$ be an integer such that $0<r<p.$ For all $n\geq 0,$ $$c\phi_p(pn+r) \equiv 0\pmod{p^2}.$$ 
\end{theorem}

Another exception to focusing on a particular value of the subscript $k$ appears in the work of Garvan and Sellers 
\cite[Theorem 2.2]{GarSel}
where the authors prove the following theorem:  

\begin{theorem}
\label{GarSelmainthm}
Let $p$ be prime and let $r$ be an integer such that $0<r<p.$  If 
$$
c\phi_k(pn+r) \equiv 0\pmod{p}
$$ 
for all $n\geq 0,$ then 
$$
c\phi_{pN+k}(pn+r) \equiv 0\pmod{p}
$$ 
for all $N\geq 0$ and $n\geq 0.$
\end{theorem}
Our goal in this note is to follow a path similar to the above theorem of Garvan and Sellers, where an infinite family of values of $k$ is identified while the value of the modulus is fixed.  

It is clear, when one reviews the literature on the subject of congruences satisfied by generalized Frobenius partition functions, that the functions $c\phi_k(n)$ satisfy many more congruences than their counterpart functions $\phi_k(n).$  One might argue that this is true ``combinatorially'' given the structure of the objects being counted by these functions (and symmetries that are inherent in the two--rowed arrays counted by $c\phi_k(n)$). It is also true that, although the generating functions for each of these two families of functions are extremely similar, the presence of certain powers of roots of unity in the generating function for $\phi_k(n),$ and the corresponding absence of such roots of unity in the generating function for $c\phi_k(n),$ may contribute to the relative lack of congruences satisfied by $\phi_k(n).$  Whatever the case, our primary goal in this note is to alter this narrative by proving the following surprising result:  

\begin{theorem}
\label{mainresult} 
Let $\ell$ be a positive integer,  $p\geq 5$ be prime, and let  $r,$ $ 0< r < p,$ be an integer such that $24r+1$ is a quadratic nonresidue modulo $p.$ For all $n\geq 0,$
 $$
 \phi_{p\ell-1}(pn+r) \equiv 0 \pmod{2}.
 $$
\end{theorem}

\section{Proof of Theorem \ref{mainresult}}

In order to prove Theorem \ref{mainresult}, we need a few preliminary facts.  First, we remind the reader of the $q$--Pochhammer symbol which is defined as follows:  
$$(A;q)_\infty = (1-A)(1-Aq)\cdots(1-Aq^{n})\cdots$$

We will also need the following two well--known results:  
\begin{theorem}
\label{eulerPNT}
\begin{equation*}
    (q;q)_\infty = \sum\limits_{k = -\infty}^\infty (-1)^k q^{\frac{3k^2-k}{2}}
\end{equation*}
\end{theorem}
\begin{proof}
See Hirschhorn 
\cite[(1.6.1)]{Hir}.  
\end{proof}

\begin{theorem}
\label{jacobi_cube} 
\begin{equation*}
    (q;q)_\infty^3 = \sum\limits_{k=0}^\infty (-1)^k(2k+1)q^{\frac{k^2+k}{2}}.
\end{equation*}
\end{theorem}
\begin{proof}
See Hirschhorn 
\cite[(1.7.1)]{Hir}.  
\end{proof}

Next, we prove an extremely important fact about the generating function for $\phi_k(n)$ for all $k\geq 1$ using a key result that appears in Andrews 
\cite{AndMem}.

\begin{theorem}
\label{Mod2Lemma} 
Let $$
\Phi_k(q) = \sum_{n=0}^\infty \phi_k(n)q^n
$$
be the generating function for the generalized Frobenius partition function $\phi_k(n).$  Then 
$$
\Phi_k(q) \equiv \frac{(q;q)_\infty}{(q^{k+1};q^{k+1})_\infty} \pmod{2}.  
$$
\end{theorem}
\begin{proof}
From Andrews' Memoir 
\cite[Theorem 7.1]{AndMem},
we know 
\begin{equation*}
    \Phi_k(q) = \frac{1}{ (q;q)_\infty^2(q^{k+1};q^{k+1})_\infty}\sum\limits_{j,r = -\infty \atop r \geq (k+1)|j|}^{\infty}(-1)^{r+kj}q^{\binom{r+1}{2}-\binom{k+1}{2}j^2}.
\end{equation*}
Therefore, we have 
\begin{eqnarray*}
\Phi_k(q) 
&=& 
\frac{1}{ (q;q)_\infty^2(q^{k+1};q^{k+1})_\infty}\sum\limits_{j,r = -\infty \atop r \geq (k+1)|j|}^{\infty}(-1)^{r+kj}q^{\binom{r+1}{2}-\binom{k+1}{2}j^2}\\
&= & 
\frac{(q;q)_\infty}{ (q;q)_\infty^3(q^{k+1};q^{k+1})_\infty}\sum\limits_{r =0}^\infty q^{\binom{r+1}{2}} \sum\limits_{|j|\leq r/(k+1)}(-1)^{r+kj}q^{-\binom{k+1}{2}j^2}  \\
&\equiv & 
\frac{(q;q)_\infty}{ (q;q)_\infty^3(q^{k+1};q^{k+1})_\infty}\sum\limits_{r =0}^\infty q^{\binom{r+1}{2}} \sum\limits_{|j|\leq r/(k+1)}q^{-\binom{k+1}{2}j^2} \pmod{2} \\
&=& 
\frac{(q;q)_\infty}{ (q;q)_\infty^3(q^{k+1};q^{k+1})_\infty}\sum\limits_{r =0}^\infty q^{\binom{r+1}{2}}\left(1+2 \sum\limits_{1\leq j \leq r/(k+1)}q^{-\binom{k+1}{2}j^2} \right)   \\
&\equiv & 
\frac{(q;q)_\infty}{ (q;q)_\infty^3(q^{k+1};q^{k+1})_\infty}\sum\limits_{r =0}^\infty q^{\binom{r+1}{2}} \pmod{2} \\
&\equiv & 
\frac{(q;q)_\infty(q;q)_\infty^3}{ (q;q)_\infty^3(q^{k+1};q^{k+1})_\infty} \pmod{2} \\
\end{eqnarray*}
thanks to Theorem \ref{jacobi_cube}.  The result immediately follows.  
\end{proof}

We are now in an excellent position to prove Theorem \ref{mainresult}.  

\begin{proof}(of Theorem \ref{mainresult}) 
Thanks to Theorem \ref{Mod2Lemma}, we know that the generating function for $\phi_{p\ell-1}$ satisfies 
\begin{eqnarray*}
\sum_{n=0}^\infty \phi_{p\ell-1}(n)q^n 
&\equiv &
\frac{(q;q)_\infty}{(q^{p\ell};q^{p\ell})_\infty} \pmod{2} \\
&\equiv & 
\frac{1}{(q^{p\ell};q^{p\ell})_\infty} \sum\limits_{k = -\infty}^\infty q^{\frac{3k^2-k}{2}} \pmod{2}  
\end{eqnarray*}
where the last statement follows from Theorem \ref{eulerPNT}.  
Since $(q^{p\ell};q^{p\ell})_\infty$ is a function of $q^p,$ and since we are interested in the parity of the values $\phi_{p\ell-1}(pn+r)$ where $0<r<p,$ we simply need to determine when $$pn+r = \frac{3k^2-k}{2}$$ for some integer $k.$  After completing the square, this is equivalent to asking whether
\begin{align*}
    24r+1 & \equiv 36k^2 -12k + 1 \pmod{p}\\
    & = (6k-1)^2.
\end{align*}
However, we assumed that $24r+1$ is a quadratic nonresidue modulo $p$ in the statement of this theorem.   Therefore, $pn+r$ can never be represented as $\frac{3}{2}k^2-\frac{1}{2}k$ for any integer $k$. This implies that 
$$\phi_{p\ell-1}(pn+r) \equiv 0 \pmod{p}.$$
\end{proof}

\section{Concluding Remarks}  
It should be noted that a companion result to Theorem \ref{mainresult}  exists for the functions $c\phi_k(n)$ as well (which is not all that surprising).  

\begin{theorem} 
\label{cphi_mod2} 
For all $k\geq 1$ and all $n\geq 0,$
 $$
 c\phi_{2k}(2n+1) \equiv 0 \pmod{2}.
 $$
\end{theorem}

In a real sense, the proof of this result already appears in the literature; see, for example, Garvan and Sellers 
\cite{GarSel}
as well as Baruah and Sarmah 
\cite{BarSar2}. 
For completeness' sake, we provide a proof here.  

\begin{proof} 
The generating function for $c\phi_{2k}(n),$ as provided by Andrews 
\cite{AndMem},
is the constant term in 
$$
CG_{2k}(z) = \prod_{n=0}^\infty(1+zq^{n+1})^{2k}(1+z^{-1}q^n)^{2k}.
$$
Thanks to the Binomial Theorem, we have 
\begin{align*}
    CG_{2k}(z) 
    &= \prod_{n=0}^\infty(1+zq^{n+1})^{2k}(1+z^{-1}q^n)^{2k}\\
    & \equiv \prod_{n=0}^\infty(1+z^2q^{2n+2})^{k}(1+z^{-2}q^{2n})^{k} \pmod{2}.
\end{align*}
The theorem immediately follows because the last product above is a function of $q^2.$
\end{proof}

\end{document}